\documentclass[12pt]{amsart}
\usepackage{amsmath,times,epsfig,amssymb,amsbsy,amscd,amsfonts,amstext,color,bm}
\usepackage{enumerate}
\usepackage{placeins}
\usepackage{array}
\usepackage{tabulary}
\usepackage{multirow}
\usepackage{graphicx}
\usepackage{hyperref}
\usepackage{hhline}
\usepackage{tabu}
\usepackage[table]{xcolor}
\usepackage{subcaption}
\usepackage[color,matrix,arrow]{xypic}
\hypersetup{
  colorlinks   = true, 
  urlcolor     = blue, 
  linkcolor    = blue, 
  citecolor   = red 
}

\theoremstyle{plain}
\newtheorem{theorem}{Theorem}[section]
\newtheorem{corollary}[theorem]{Corollary}
\newtheorem{lemma}[theorem]{Lemma}

\newtheorem{proposition}[theorem]{Proposition}

\newtheorem{problem}[theorem]{Problem}

\makeatletter
    
    \@addtoreset{equation}{section}
  \makeatother

\theoremstyle{definition}
\newtheorem{definition}[theorem]{Definition}
\newtheorem{example}[theorem]{Example}

\theoremstyle{remark}
\newtheorem{remark}[theorem]{Remark}

\newcommand{\A}{\mathcal{A}}
\newcommand{\scB}{\mathcal{B}}
\newcommand{\C}{\mathbb{C}}
\newcommand{\F}{\mathbb{F}}
\newcommand{\scF}{\mathcal{F}}
\newcommand{\scL}{\mathcal{L}}

\newcommand{\R}{\mathbb{R}}

\newcommand{\G}{\mathcal{G}}

\newcommand{\scS}{\mathcal{S}}
\newcommand{\scT}{\mathcal{T}}

\newcommand{\scP}{\mathcal{P}}

\newcommand{\Z}{\mathbb{Z}}

\newcommand{\scH}{{\mathcal{H}}}

\newcommand{\M}{\mathcal{M}}

\newcommand{\lcm}{\operatorname{lcm}}
\newcommand{\Hom}{\operatorname{Hom}}
\newcommand{\quasi}{\operatorname{quasi}}

\renewcommand{\part}{\operatorname{par}}

\newcommand{\tor}{\operatorname{tor}}

\newcolumntype{K}[1]{>{\centering\arraybackslash}p{#1}}

\begin{document}

\title[An equivalent formulation of chromatic quasi-polynomials]{An equivalent formulation of chromatic quasi-polynomials}

\date{\today}

\begin{abstract}
Given a central integral arrangement, the reduction of the arrangement modulo positive integers $q$ gives rise to a subgroup arrangement in $(\mathbb{Z}/q\mathbb{Z})^\ell$.
Kamiya-Takemura-Terao (2008) introduced the notion of characteristic quasi-polynomials, which uses to evaluate the cardinality of the complement of the subgroup arrangement.
Chen-Wang (2012) found a similar but more general setting that replacing the integral arrangement by its restriction to a subspace of $\mathbb{R}^\ell$, and evaluating the cardinality of the $q$-reduction complement will also lead to a quasi-polynomial in $q$. 
On an independent study, Br{\"a}nd{\'e}n-Moci (2014) defined the so-called chromatic quasi-polynomial, and initiated the study of $q$-colorings on a finite list of elements in a finitely generated abelian group. 
The main purpose of this paper is to verify that the Chen-Wang's quasi-polynomial and the Br{\"a}nd{\'e}n-Moci's chromatic quasi-polynomial are equivalent in the sense that the quasi-polynomials enumerate the cardinalities of isomorphic sets.
 \end{abstract}

\author{Tan Nhat Tran}
\address{Tan Nhat Tran, Department of Mathematics, Hokkaido University, Kita 10, Nishi 8, Kita-Ku, Sapporo 060-0810, Japan.}
\email{trannhattan@math.sci.hokudai.ac.jp}




\date{\today}
\maketitle


\section{Introduction}
\label{sec:intro}
 \textbf{Background}.
In the simplest setting, when a finite list $\A$ of integer vectors in $\Z^\ell$ is given, we may naturally associate to it a central hyperplane arrangement $\A(\R)$ in $\R^\ell$, which we call \emph{integral arrangement}. 
The study of a hyperplane arrangement typically goes along with the study of its \emph{characteristic polynomial} as the polynomial carries combinatorial and topological information of the arrangement (e.g., \cite{OS80}). 
In this paper, we are mainly interested in an arithmetical method, generally known as ``finite field method", for studying the integral arrangements. 
The method probably was first initiated by \cite{CR70} and developed into a systematic tool by Athanasiadis \cite{A96}, after closely related techniques have been used by Bj{\"o}rner-Ekedahl \cite{BE97}, and Blass-Sagan \cite{BS98} to solve problems related to subspace arrangements. 
Roughly speaking, suppose that the integral arrangement $\A(\R)$ associated to the list $\A$ is given, we can take coefficients modulo a positive integer $q$ and get an arrangement $\A(\Z/q\Z)$ of subgroups in $(\Z/q\Z)^\ell$. 
One of the reasons why the method is regarded as the ``finite field method" presumably comes from one of the most well-known and fundamental results in the theory. 
It states that when $q$ is a sufficiently large prime, the arrangement $\A(\Z/q\Z)$ now is defined over the finite field $\F_q$, and the cardinality of its complement coincides with $\chi_{\A(\R)}(q)$, the evaluation of the characteristic polynomial $\chi_{\A(\R)}(t)$ of $\A(\R)$ at $q$ (e.g., \cite[Theorem 2.2]{A96}). 

The fundamental theorem mentioned above is efficiently applicable to compute the characteristic polynomials of several arrangements arising from root systems (e.g., \cite{A96}).  
Kamiya-Takemura-Terao \cite{KTT08} showed that the cardinality of the complement is actually a quasi-polynomial in $q$, and named this the \emph{characteristic quasi-polynomial} of $\A$ as its $1$-constituent is identical with $\chi_{\A(\R)}(t)$. 
Chen-Wang \cite{CW12} considered the restriction of the integral arrangement to a subspace of $\R^\ell$, and proved a stronger result that after taking reduction modulo $q$ of the restricted arrangement, the cardinality of the complement of is also a quasi-polynomial in $q$. 
Later on, Yoshinaga (\cite{Y16}, \cite{Yo16}) extended the analysis on the deformations of root system arrangements and enhanced the calculation of the characteristic quasi-polynomials via the connection with Ehrhart quasi-polynomials.

In yet another consideration, given a finite list $\A$ in $\Z^\ell$, we can also associate it a \emph{toric arrangement} $\A(G)$ in the torus $G^\ell$ with $G$ is ${\Bbb S}^1$ or $\C^\times$ (e.g.,  \cite{DP05}, \cite{L12}). Although we are concerned with the subtori (hypersurfaces) instead of the hyperplanes, the ``finite field method" remains alive and is formulated in various ways (e.g., \cite{Law11}, \cite{ERS09}, \cite{ACH15}). 
To compute and even to make a broader understanding the characteristic polynomial, an arithmetical generalization of the ordinary \emph{Tutte polynomial} \cite{T54}, the \emph{arithmetic} Tutte polynomial was introduced \cite{L12}. 
These polynomials are currently receiving increasing attention (e.g., \cite{DM13}, \cite{FM16}). 
Br{\"a}nd{\'e}n-Moci \cite{BM14} defined the \emph{Tutte quasi-polynomial} associated to a finite list of elements in a finitely generated abelian group. 
This quasi-polynomial not only produces the interpolation between the Tutte polynomial and the arithmetic Tutte polynomial but also gives rise to \emph{chromatic quasi-polynomial} and \emph{flow quasi-polynomial}. 
These are group-theoretical counterparts of the graphic chromatic and flow polynomials which proved to have an application to colorings and flows on CW complexes \cite{DM16}.

 \textbf{Our result}. 
 A newly introduced notion of  \emph{$G$-Tutte polynomials} \cite{LTY17} establishes a common generalization of several ``Tutte-like" polynomials including all of the (quasi-)polynomials mentioned previously. 
In particular, the notion of $G$-Tutte polynomials is useful that enables us to unify the objects. 
We take the advantage to get the result that the Chen-Wang's quasi-polynomial and the Br{\"a}nd{\'e}n-Moci's chromatic quasi-polynomial are ``the same" in the sense that any Chen-Wang's quasi-polynomial is a chromatic quasi-polynomial and vice versa (equalities \eqref{eq:CW} and \eqref{eq:CW=BM}).

 \textbf{Organization of the paper}. 
The remainder of the paper is organized as follows. 
In Section \ref{sec:background}, we recall definitions of the characteristic, Chen-Wang's quasi-polynomials (\S\ref{subsec:CQP}), the chromatic quasi-polynomials (\S\ref{subsec:ChQP}) and briefly recall the motivations of why they have been defined. 
In Section \ref{sec:unify}, after recalling the basic facts of $G$-Tutte polynomials, we show the equivalence of the chromatic quasi-polynomials and the Chen-Wang's quasi-polynomials (equalities \eqref{eq:CW} and \eqref{eq:CW=BM}).
In \S\ref{subsec:Chromatic}, we give a Deletion-Contraction formula (Theorem \ref{thm:DC-formula}) for the chromatic quasi-polynomials. 
By using the language of chromatic quasi-polynomials, we also give a discussion to a problem asked by Chen-Wang (Problem \ref{pro:CW}).
In Section \ref{sec:real-arr}, we generalize the fundamental theorem in the primary ``finite field method" applied to any $\R$-plexification. 
That is, the $1$-constituent of the chromatic quasi-polynomial of a list $\A$ either is $0$ or agrees with the characteristic polynomial of the corresponding $\R$-plexification $\A(\R)$, and the characteristic polynomial of any $\R$-plexification $\A(\R)$ can be computed by the $1$-constituent of the chromatic quasi-polynomial of the deletion list of $\A$ by the list of its torsion elements (Theorem \ref{thm:quasi-general} and Proposition \ref{prop:non-torsion}).

\medskip

\section{Preliminaries} 
\label{sec:background}
Let us first fix some definitions and notations throughout the paper. 

A function $f: \Z \to \Z$ is called a \emph{quasi-polynomial} if there exist $\rho\in\Z_{>0}$ and polynomials $g^k(t)\in\Z[t]$ ($1 \le k \le \rho$) such that for any $q\in\Z_{>0}$, 
\begin{equation*}
f(q) =g^k(q), 
\end{equation*}
when $q\equiv k\bmod \rho$. The number $\rho$ is called a period and the polynomial $g^k(t)$ is called the \emph{$k$-constituent} of $f(q)$. 

Let $\Gamma$ be a finitely generated abelian group, and
let $\A\subseteq\Gamma$ be a finite list (multiset) of elements in $\Gamma$. 
We will use the term \emph{pair} $(\A, \Gamma)$ to refer to these objects. 

For each sublist $\scS\subseteq \Gamma$, we denote by $r_\scS$ the \emph{rank} (as an abelian group)
of the subgroup $\langle\scS\rangle\le \Gamma$ generated by $\scS$. 
By the Structure Theorem, we may write $\Gamma/\langle\scS\rangle\simeq\bigoplus_{i=1}^{n_{\scS}}\Z/d_{\scS, i}\Z\oplus\Z^{r_{\Gamma}-r_{\scS}}$ where $n_{\scS}\geq 0$ and $1<d_{\scS, i}|d_{\scS, i+1}$. 
The \emph{LCM-period} $\rho_{\A}$ of $\A$ is defined by 
\begin{equation*}
\label{eq:LCM-period}
\rho_\A:=\lcm(d_{\scS, n_{\scS}}\mid\scS\subseteq \A). 
\end{equation*} 

Given a group $K$, denote by $K_{\tor}$ the torsion subgroup of $K$. Denote 
$$\scS^{\tor}:= \scS \cap \Gamma_{\tor} .$$

We are going to investigate one of typical problems in enumerative combinatorics: counting the sizes of sets which depend on positive integers $q$. 
Our motivated example is the graphic \emph{chromatic polynomials}.
Let $\G = (V,E)$ be a graph. 
Enumerating the set  $c_{\G}(q)$ of all proper $q$-colorings, i.e., labelings $\textbf{x} \in \{1,\ldots, q\}^{\#V}$ such that adjacent vertices get different labels: if $(ij) \in E$ then $x_i \ne x_j$, gives rise to a polynomial in $q$. 
The polynomial $c_{\G}(q)$ is broadly known as the chromatic polynomial of $\G$,  going back to Birkhoff and Whitney.
More generally, it happens quite often that enumerating the cardinalities of sets will lead to quasi-polynomials. 
One of the most famous examples is that given a rational polytope $\scP \subseteq \R^d$, the function $\#(q\scP \cap \Z^d)$ for $q\in\Z_{>0}$ agrees with a quasi-polynomial, called the \emph{Ehrhart quasi-polynomial}.

\subsection{Characteristic and Chen-Wang's quasi-polynomials} 
\label{subsec:CQP}
Our next and important example that a quasi-polynomial appears in the counting problem list is that of characteristic quasi-polynomials. 
We will define the characteristic quasi-polynomials in a slightly different language to what has been stated in the Introduction part. 
For instance, the arrangement $\A(\Z/q\Z)$ called  \emph{$q$-reduction arrangement} in \cite{KTT08} and its complement will not appear here but later in Section \ref{subsec:G-Tutte} after invoking the notion of $G$-plexifications.
We specify $\Gamma=\Z^\ell$. 
Let $q\in\Z_{>0}$, and set $(\Z/q\Z)^\times := (\Z/q\Z)\smallsetminus\{\overline0\}$.
For simplicity of notation, we use the same symbols $\A$ and $\textbf{z}$ for the realizations of the list $\A\subseteq\Z^\ell$ and the element $\textbf{z} \in  (\Z/q\Z)^\ell $ as matrices of size $\ell \times \#\A$ and $1 \times \ell$, respectively. 
Denote
$${\mathrm{KTT}}(\A, \Z^\ell; q):= \{ \textbf{z} \in  (\Z/q\Z)^\ell \mid \textbf{z}\cdot\A \in ((\Z/q\Z)^\times)^{\#\A}\}.$$
We agree that ${\mathrm{KTT}}(\emptyset, \Z^\ell; q) = (\Z/q\Z)^\ell$, thought of as no constraints on $\textbf{z}$.
Kamiya-Takemura-Terao \cite{KTT08} produced two different methods with one of them relies on the theory of Ehrhart quasi-polynomials to show that $\#{\mathrm{KTT}}(\A, \Z^\ell; q)$ is a monic quasi-polynomial in $q$ with a period $\rho_{\A}$.  
The quasi-polynomial is called the \emph{characteristic quasi-polynomial} of $\A$. 
The name really explains the main reason of why the quasi-polynomial was introduced as its generality influences the study of real hyperplane arrangements.
That is, its $1$-constituent coincides with the characteristic polynomial (e.g., \cite[Definition 2.52]{OT92}) of the hyperplane arrangement $\{H_{\alpha} \mid \alpha  \in \A\}$ in $\R^\ell$ with $H_{\alpha}$ is the hyperplane orthogonal to $\alpha$ (e.g.,  \cite{A96}, \cite{KTT08}).

Let $\scB$ be another finite list in $\Z^\ell$. 
Chen-Wang  \cite{CW12}  considered a more general setting 
\begin{equation*}
\label{eq:CW-def}
{\mathrm{CW}}(\A, \scB, \Z^\ell; q):=\left\{
\textbf{z}\in  (\Z/q\Z)^\ell
\left|
\begin{array}{cc}
\textbf{z}\cdot \A \in ((\Z/q\Z)^\times)^{\#\A} \\
\textbf{z}\cdot \scB = (\overline0)^{\#\scB}
\end{array}
\right.
\right\}, 
\end{equation*}
and applied the elementary divisor method of \cite{KTT08} to show that the cardinality $\#{\mathrm{CW}}(\A, \scB, \Z^\ell; q)$ is also a quasi-polynomial in $q$. 
The notion of Chen-Wang's quasi-polynomials strictly generalizes that of characteristic quasi-polynomials because ${\mathrm{KTT}}(\A, \Z^\ell;q)={\mathrm{CW}}(\A, \scB, \Z^\ell; q)$ when $\scB$ is the zero matrix, and $\#{\mathrm{KTT}}(\emptyset, \Z^\ell;q)=q^\ell$ while $\#{\mathrm{CW}}(\emptyset, \scB, \Z^\ell; q)$ still depends on $\scB$.

\subsection{Chromatic quasi-polynomials} 
\label{subsec:ChQP} 
Let $(\A, \Gamma)$ be any pair. 
Br{\"a}nd{\'e}n-Moci \cite{BM14} defined the following set 
$${\mathrm{BM}}(\A, \Gamma; q):= \{ \varphi \in\Hom(\Gamma, \Z/q\Z)\mid \varphi(\alpha) \ne \overline0 \mbox{ for all }\alpha\in\A\},$$
and proved that its cardinality $\#{\mathrm{BM}}(\A, \Gamma; q)$ is a quasi-polynomial in $q$ for which $\rho_{\A}$ is a period.\footnote{Br{\"a}nd{\'e}n-Moci actually defined a somewhat different period. However, the fact that $\rho_{\A}$ is also a period becomes clear after proving Theorem \ref{thm:gen-quasi}.} 
Thus any characteristic quasi-polynomial is indeed a Br{\"a}nd{\'e}n-Moci's quasi-polynomial by the following way.
Fix a standard basis (of unit vectors) $\{\epsilon_1, \ldots, \epsilon_\ell\}$ for $\Z^\ell$, and apply the isomorphism $\Hom(\Z^\ell,  \Z/q\Z) \simeq (\Z/q\Z)^\ell$ to obtain ${\mathrm{KTT}}(\A, \Z^\ell;q)={\mathrm{BM}}(\A,  \Z^\ell; q)$.

The authors named the quasi-polynomial the \emph{chromatic quasi-polynomial} as it generalizes the concept of chromatic polynomials defined on graphs. 
We briefly recall how it can be seen. 
Let $\G = (V,E)$ be a graph. Define a list of vectors $\scL = \{\alpha_e \mid e \in E\}$ in $\Z^{\#V}$ as follows. 
If $e=(ij) \in E$, let $\alpha_e$ be the vector with entry $j$ is $1$, entry $i$ is $-1$, and the other entries are $0$. 
Then the chromatic polynomial $c_{\G}(q)$ of $\G$ can be expressed as $c_{\G}(q)=\#{\mathrm{BM}}(\scL, \Z^{\#V}; q)$ with the unique constituent coincides with the characteristic polynomial of the real \emph{graphical arrangement} $\{\{x_i=x_j\} \mid (ij) \in E\}$ in variable $q$ (e.g., \cite[Theorem 2.88]{OT92}).
\section{Unify the quasi-polynomials} 
\label{sec:unify}

\subsection{$G$-Tutte polynomials}
\label{subsec:G-Tutte} 
The Chen-Wang's quasi-polynomial and the Br{\"a}nd{\'e}n-Moci's chromatic quasi-polynomial arise independently in different contexts and may seem unrelated at first glance. 
We will show that the notion of $G$-Tutte polynomials is useful to unify them.

Let $G$ be an arbitrary abelian group.
We recall the notions of $G$-plexifications and $G$-Tutte polynomials of $\A$ following \cite[\S3]{LTY17}. 
We regard $\Hom(\Gamma, G)$ as our total group. 
For each $\alpha \in \A$, we define the \emph{$G$-hyperplane} associated to $\alpha$ as follows:
\begin{equation*}
H_{\alpha, G}:=\{\varphi \in\Hom(\Gamma, G)\mid \varphi (\alpha)= 0\}. 
\end{equation*}
Then the \emph{$G$-plexification} $\A(G)$ of $\A$ is the collection of the subgroups $H_{\alpha, G}$
$$\A(G):=\{H_{\alpha, G}\mid\alpha\in\A\}.$$
The \emph{$G$-complement} $\M(\A; \Gamma, G)$ of $\A(G)$ is defined by
\begin{equation*}
\M(\A; \Gamma, G):=\Hom(\Gamma, G)\smallsetminus\bigcup_{\alpha\in\A}H_{\alpha, G}. 
\end{equation*}

 For any sublist $\scS\subseteq \A$, the \emph{deletion} $\A\smallsetminus\scS$ is defined as a list of elements in the same group $\Gamma$. 
We also define the \emph{contraction} $\A/\scS$ as the list of cosets $\{\overline{\alpha}\mid 
\alpha\in\A\smallsetminus\scS\}$ in the group $\Gamma/\langle\scS\rangle$. 
The method of identifying sets discussed in \cite[\S3.2]{LTY17} enables us to write 
\begin{equation}
 \label{eq:T&chi}
\M(\A/\scS; \Gamma/\langle\scS\rangle, G)=\left\{
\varphi\in\Hom(\Gamma, G)
\left|
\begin{array}{cc}
\varphi(\alpha)=0, &\mbox{for all }\alpha\in\scS\\
\varphi(\alpha)\neq 0, &\mbox{for all }\alpha\in\A\smallsetminus\scS
\end{array}
\right.
\right\}
\end{equation}

An abelian group $G$ is said to be \emph{torsion-wise finite} if $G[d]:=\{x\in G\mid d\cdot x=0\}$ is finite for all $d\in\Z_{>0}$. 
In what follows, we assume that $G$ is torsion-wise finite. 
 \begin{definition}
\label{def:multiplicit}
The \emph{$G$-multiplicity} $m(\scS; G)$ for each $\scS\subseteq \A$ is defined by 
\begin{equation*}
m(\scS; G):=
\#\Hom\left((\Gamma/\langle\scS\rangle)_{\tor}, G\right). 
\end{equation*}
\end{definition}

\begin{definition}\quad
\label{def:main}
\begin{enumerate}[(1)]
\item  
The \emph{$G$-Tutte polynomial} $T_{\A}^{G}(x, y)$ of $\A$ is defined by 
\begin{equation*}
T_{\A}^{G}(x, y):=
\sum_{\scS\subseteq \A}m(\scS; G)(x-1)^{r_\A-r_\scS}(y-1)^{\#\scS-r_\scS}. 
\end{equation*}
\item 
The \emph{$G$-characteristic polynomial} $\chi_{\A}^G(t)$ of $\A$ is defined by 
\begin{equation*}
\chi_{\A}^G(t):=(-1)^{r_\A}\cdot t^{r_{\Gamma}-r_\A}\cdot T_{\A}^{G}(1-t, 0). 
\end{equation*}
\end{enumerate}
\end{definition}

The notion of $q$-reduction arrangements we mentioned in \S\ref{subsec:CQP} is specialization of that of $\Z/q\Z$-plexifications.
For general $\Gamma$, it turns out that 
\begin{equation}
\label{eq:BM=G-Tutte}
{\mathrm{BM}}(\A, \Gamma; q)=\M(\A; \Gamma, \Z/q\Z).
\end{equation} 

Using formula \eqref{eq:T&chi} and the equality \eqref{eq:BM=G-Tutte} above, we can write
\begin{equation}
\label{eq:CW}
{\mathrm{CW}}(\A, \scB, \Z^\ell; q)={\mathrm{BM}}((\A\sqcup\scB)/\scB, \Z^\ell/\langle\scB\rangle; q). 
\end{equation} 
Thus any Chen-Wang's quasi-polynomial is a chromatic quasi-polynomial defined on a certain contraction list. 
The converse is also true as we will see in the lemma below.
\begin{lemma}
\label{lem:useful-construction}
 Given a pair $(\A, \Gamma)$ with $\Gamma\simeq \Z^r \oplus\Z/d_1\Z\oplus\cdots\oplus\Z/d_s\Z$, we can find two lists $Q \subseteq \scL \subseteq \Z^{r+s}$ with $r_Q=s$ such that $\A = \scL/Q$. 
\end{lemma}
 \begin{proof} 
 We can view $\Gamma\simeq\Z^{r+s}/\langle Q \rangle$, where $Q=\{q_1,\ldots, q_s\} \subseteq \Z^{r+s}$, $q_i$ has $d_i$ in the $(r + i )$-th coordinate and 0 elsewhere. 
Thus $\A$ can be identified with a list of cosets $\A= \{\overline a _1,\ldots, \overline a _k\}$ with $a_i \in \Z^{r+s}$. 
We choose a representative $a _i \in \Z^{r+s}$ for each coset, which is determined up to a linear combination of elements from $Q$.
Define $\tilde{\A}:= \{a _1,\ldots,  a _k\}\subseteq\Z^{r+s}$, and $\scL := \tilde{\A}\sqcup Q \subseteq \Z^{r+s}$. 
Thus $\A=\scL/Q$. 
\end{proof}

\begin{remark}
\label{rem:useful-construction}
The construction of the lists $\tilde{\A}$ and $\scL$ presented in Lemma \ref{lem:useful-construction} is probably well-known among experts, for instance \cite[\S3.4]{DM13}, wherein it plays a crucial role in proving the representability of the duals of arithmetic matroids.
\end{remark}

With the notation as in Lemma \ref{lem:useful-construction}, for any pair $(\A, \Gamma)$ we can write 
\begin{equation}
\label{eq:CW=BM}
{\mathrm{BM}}(\A, \Gamma; q)={\mathrm{CW}}(\tilde{\A}, Q, \Z^{r+s}; q).
\end{equation}

We have verified that the Chen-Wang's quasi-polynomial and the Br{\"a}nd{\'e}n-Moci's chromatic quasi-polynomial are ``equivalent'' in the sense that the quasi-polynomials enumerate the cardinalities of isomorphic sets.

\subsection{More on chromatic quasi-polynomials}
\label{subsec:Chromatic}
Given any pair $(\A, \Gamma)$, let us denote by $\chi^{\quasi}_{\A}(q)$ the chromatic quasi-polynomial of $\A$ i.e., $\chi^{\quasi}_{\A}(q)=\#\M(\A; \Gamma, \Z/q\Z)$.
We also write $f_{\A}^k(t)$ for the $k$-constituent of $\chi^{\quasi}_{\A}(q)$ ($1 \le k \le \rho_\A$). 
\begin{theorem}[\cite{BM14}, \cite{LTY17}]
\label{thm:gen-quasi}
$\chi^{\quasi}_{\A}(q)=\chi_{\A}^{\Z/q\Z}(q).$
\end{theorem}

\begin{proposition}[\cite{CW12}]
\label{prop:gen-quasi-gcd2}\quad
\begin{enumerate}[(1)]
\item For any $k$ with $1 \le k \le\rho_{\A}$,  $f^k_{\A}(t)=\chi_{\A}^{\Z/k\Z}(t)$. 
\item $\chi^{\quasi}_{\A}(q)$ satisfies the \emph{GCD-property} i.e. $f_{\A}^a(t)=f_{\A}^{b}(t)$ if $\gcd(a, \rho_\A)=
\gcd(b, \rho_\A)$. 
\item For any $k$ with $1 \le k \le\rho_{\A}$, if $\gcd(q, \rho_\A)=k$, then $\chi^{\quasi}_{\A}(q)=f^k_{\A}(q)$.
\end{enumerate}
\end{proposition}

Fix $\alpha\in\A$. Denote $\A':=\A\smallsetminus\{\alpha\}$, and $\A'':=\A/\{\alpha\}$.
\begin{theorem}[Deletion-Contraction formula]
\label{thm:DC-formula}
$$\chi^{\quasi}_{\A}(q)=\chi^{\quasi}_{\A'}(q)-\chi^{\quasi}_{\A''}(q).$$
\end{theorem}
 \begin{proof}
This follows directly from \cite[Corollary 4.11]{LTY17} and Theorem \ref{thm:gen-quasi} by letting $G=\Z/q\Z$. 
We can also obtain it from \cite[Proposition 3.4]{LTY17}.
\end{proof}

\begin{remark}
\label{rem:DR-formula}
Using Theorem \ref{thm:DC-formula}, the Deletion-Restriction formula in \cite[Lemma 3.3]{CW12} can be exhibited by setting $\A$ as the contraction list $(A\sqcup B)/B$, where $A \ne \emptyset$ and $B$ are finite lists in $\Z^\ell$.
\end{remark}

\begin{corollary} 
\label{cor:DC-const}
 If $k\le\min\{\rho_{\A'},\rho_{\A''}\}$, then the $k$-constituents satisfy
 $$f^k_{\A}(t)=f^k_{\A'}(t)-f^k_{\A''}(t).$$ 
 \end{corollary}
\begin{proof}
Note that the LCM-period of any deletion/contract list is a divisor of the LCM-period of the parent list. 
\end{proof}
 
\begin{remark}
\label{rem:non-minimum}
For a pair $(\A, \Gamma)$, the LCM-period of $\chi^{\quasi}_{\A}(q)$ is not necessarily the minimum period. 
We clarify it by an example. 
Let $\Gamma=\Z/2\Z\oplus\Z/2\Z$, $\A=\{\alpha, \beta\}\subsetneq\Gamma$ with 
$\alpha=(\overline{0}, \overline{0})$ and $\beta=(\overline{1}, \overline{0})$. 
Then $\rho_\A=2$, while the minimum period is actually 1 and $\chi^{\quasi}_{\A}(q) = 0$ for every $q$.
Note that this fact can also be clarified by another class of examples originated from  \cite[Example 4.2]{CW12}.
\end{remark}

We close this section by giving a discussion on a problem asked in \cite[Problem 2]{CW12}.
\begin{problem}
\label{pro:CW}
Let $\A_1, \A_2$ be finite lists in $\Z^\ell$ with $r_{\A_2}=\ell$. 
 Assume that 
 $\#{\mathrm{CW}}(\A_1, \A_2, \Z^\ell; q)=0$ for every $q \in \Z_{>0}$. Then there exists $\alpha \in \A_1$ such that $\alpha \in \langle \A_2 \rangle$.

 \end{problem}
\begin{proof}[Discussion]
The statement is true if and only if $\ell=1$. 
Assume that $\ell=1$.
By equality \eqref{eq:CW}, we rewrite the assumption as $\#{\mathrm{BM}}(\A, \Gamma; q)=0$ with $\A = (\A_1\sqcup \A_2)/ \A_2$, and $\Gamma=\Z/\langle\A_2\rangle \simeq \Z/d\Z$ for some $d \in \Z_{>0}$. 
Suppose to the contrary that for every $\alpha \in \A_1$, $\alpha \notin \langle \A_2 \rangle$. 
It is equivalent to saying that $\overline{a}  \ne \overline{0}$ for all $\overline{a} \in \A$.
Set $T:=\{z \in \C \mid z^d=1\}$.
For each $\overline{a} \in \A$ with $0 \le a \le d-1$, set $T_a:=\{z\in T \mid z^a=1\}$. 
Thus
$$f^{\rho_\A}_{\A}(t)=\chi_{\A}^{\Z/\rho_\A\Z}(t)=\# \left( T\smallsetminus\bigcup_{\overline{a}\in\A}T_a\right)>0,$$
which is a contradiction. 
For $\ell \ge 2$, we show that the statement is not true by providing a counterexample.
Let us first prove the following fact: if $\Gamma=\Z/d_1\Z\oplus\cdots\oplus\Z/d_\ell\Z$ is a finite abelian group containing at least two distinct nonidentity elements of order 2, say $\beta_1,\beta_2$, and $\A=\{\alpha \in \Gamma \mid \alpha \ne 0_{\Gamma}\}$, then $\#{\mathrm{BM}}(\A, \Gamma; q)=0$ for every $q \in \Z_{>0}$. Indeed by definition,
\begin{align*}
{\mathrm{BM}}(\A, \Gamma; q)
&=  \{ \varphi \in \Hom(\Gamma,  \Z/q\Z) \mid \varphi(\alpha) \ne \overline{0},\,\mbox{for all }\alpha\in\A \}, \\
&= \{ \varphi \in \Hom(\Gamma,  \Z/q\Z) \mid \varphi \mbox{ is injective} \} .
\end{align*}
If the set above is nonempty, then $\varphi(\alpha)$, $\varphi(\beta)$ are distinct and both have order 2 in $\Z/q\Z$. 
This contradiction implies that $\#{\mathrm{BM}}(\A, \Gamma; q)=0$. 
With the notation as in equality \eqref{eq:CW=BM}, $\#{\mathrm{CW}}(\tilde{\A}, Q, \Z^{\ell}; q)=0$.
Now choose $\Gamma=(\Z/2\Z)^\ell$ with $\ell\ge 2$, and let $\A_1=\tilde{\A}$, $\A_2=Q$.

\end{proof}

\section{Application to hyperplane arrangements}
\label{sec:real-arr} 
The aim of this section is to generalize the result in  \cite{KTT08} that the $1$-constituent of $\chi^{\quasi}_{\A}(q)$ agrees with the characteristic polynomial of some real arrangement. A natural choice is $\A(\R)$. 
However, as long as $\Gamma$ is any finitely generated abelian group and the list $\A$ may contain torsion elements of $\Gamma$, we need to know what $\A(\R)$ is all about. 
It turns out that we can realize $\A(\R)$ as a certain (restriction of) integral arrangement. 
Then we also would like to compute the characteristic polynomial of any $\R$-plexification $\A(\R)$. 
These facts will be made clear in Theorem \ref{thm:quasi-general} and Propositions \ref{prop:no-loops}, \ref{prop:non-torsion}.
More generally, we will give other interpretations for every chromatic quasi-polynomial and their constituents through subspace and toric viewpoints in our forthcoming paper \cite{TY18}.

In the following setting and until before Proposition \ref{prop:no-loops}, we restrict our attention to the case $\Gamma =\Z^\ell$, and view $\A$ as a finite list of \emph{nonzero} vectors in $\Z^\ell$. 
We regard $\{\epsilon_1, \ldots, \epsilon_\ell\}$ as the standard basis for $\R^\ell$, and equip to it the standard inner product $(\cdot,\cdot)$. 
Then the $\R$-plexification $\A(\R)$ is an arrangement of (possibly repeated) hyperplanes in $\R^\ell$ with each hyperplane $H_{\alpha, \R}$ can be identified with $H_\alpha=\{x\in \R^\ell \mid (\alpha,x)=0\}$. 
Such $\R$-plexifications are integral arrangements. 
Let $L_{\A(\R)}$ be the \emph{intersection poset} (e.g., \cite[\S2.1]{OT92}) of $\A(\R)$. 
Note that we require the intersection poset to be a set, not multiset. 
Also, the ambient space $\R^\ell$ can be added to the arrangement without affecting the arrangement's intersection poset.
For each $X \in L_{\A(\R)}$, the \emph{localization} of $\A(\R)$ on $X$ is defined by 
$$\A(\R)_X := \{ H \in \A(\R) \mid X \subseteq H\},$$
and the \emph{restriction} ${\A(\R)}^{X}$ of $\A(\R)$ to $X$ is defined by 
$${\A(\R)}^{X}:= \{ H \cap X\mid H   \in \A(\R)\smallsetminus \A(\R)_X\}.$$
Denote by $X^\perp$ the orthogonal complement of $X$ in $\R^\ell$. 
Set 
$$\A_X:=\A \cap X^\perp \subseteq \A.$$

\begin{proposition}
\label{prop:gen-quasi-pro}
The following formulas are valid at level of multisets:
\begin{enumerate}[(1)]
\item $\A(\R)_X=(\A_X)(\R)$.
\item $\A(\R)^X=(\A/\A_X)(\R)$.
\end{enumerate}

\end{proposition}
\begin{proof} 
The proof of (1) is straightforward. 
To prove (2), for every $X \in L_{\A(\R)}$ with $X \ne \R^\ell$, we use $X= \bigcap_{H\in\A(\R)_X} H$, the longest expression of $X$ in terms of intersection of the hyperplanes in $\A(\R)$.
To see $\A(\R)^X=(\A/\A_X)(\R)$ as multisets, note that the number of occurrences of each element $H_{\beta, \R}\cap X$ in these multisets is equal to $\#\{\gamma \in \A\smallsetminus \A_X \mid \gamma \in \mathrm{span}_{\R}\{\beta, \A_X\}\}$. 
\end{proof}

Denote by $\chi_{\scH}(t)$ the characteristic polynomial of the real arrangement $\scH$ (e.g., \cite[Definition 2.52]{OT92}). 
The following result is essentially appeared in \cite[Corollary 2.4]{CW12} (see also \cite[Corollary 6.1]{A96}). 
The idea of the proof is to use Whitney's theorem (e.g., \cite[Theorem 2.4]{St07}) and Proposition \ref{prop:gen-quasi-pro}.
\begin{lemma}
\label{lem:CP-restriction}
$\chi_{\A(\R)^X}(t)=f^1_{\A/\A_X}(t).$
\end{lemma}

Now we give an arrangement theoretic realization for $\A(\R)$.
\begin{proposition}
\label{prop:no-loops}
Given a pair $(\A, \Gamma)$, if $\A^{\tor} = \emptyset$ then $\A(\R)$ is an integral arrangement, and also can be realized as a restriction of $\scL(\R)$ where $\scL$ is a finite list in some $\Z^\ell$.
\end{proposition}
\begin{proof} 
We use the notation as in Lemma \ref{lem:useful-construction}. 
Set $X:=\bigcap_{q_i \in Q}H_{q_i, \R} \in L_{\scL(\R)}$, then $Q=\scL\cap X^\perp=\scL_X$. 
The condition $\A^{\tor} = \emptyset$ is crucial, otherwise it may happen that $Q\subsetneq\scL_X$.
By Proposition \ref{prop:gen-quasi-pro}, $\A(\R) = (\scL/\scL_X)(\R) =\scL(\R)^X$. 
This means that $\A(\R)$ is the restriction of $\scL(\R)$ to $X$, and also can be identified with an integral arrangement in $\R^{r_\Gamma}$.
\end{proof}

Next, we prove an important property of $\chi^{\quasi}_{\A}(q)$, which is the main theorem of this section.
\begin{theorem}
\label{thm:quasi-general}
Let $(\A, \Gamma)$ be any pair. Then
$$\chi_{\A(\R)}(t)=f^1_{\A\smallsetminus \A^{\tor}}(t).$$ 
\end{theorem}
 \begin{proof}
If $\A^{\tor} = \emptyset$, we apply Proposition \ref{prop:no-loops} and Lemma \ref{lem:CP-restriction}. 
If $\A^{\tor} \ne \emptyset$, note that $\A(\R)$ and $(\A\smallsetminus \A^{\tor})(\R)$ have the same intersection poset. 
\end{proof}

The $1$-constituent $f^1_{\A}(t)$ sometimes can be regarded as the chromatic polynomial defined on a graph, for example, via connection with graphical arrangements discussed in \S\ref{subsec:ChQP}. 
It is well known (and easy to show) that the graphic chromatic polynomial is identical to 0 if the graph contains some (graph theoretic) loop. Recall from \cite[\S 4.4]{DM13} that an element $\alpha\in\A$ is called a \emph{loop} (resp. \emph{coloop}) 
if $\alpha\in\Gamma_{\tor}$ (resp. $r_\A=r_{\A'}+1$). 
An element $\alpha\in\A$ that is neither a loop nor a coloop is said to be \emph{proper}. 
We will prove in the proposition below that a similar result holds for $f^1_{\A}(t)$.

\begin{proposition}
\label{prop:non-torsion}
Let $(\A, \Gamma)$ be any pair with $\A^{\tor} \ne \emptyset$. 
Then $$f^1_{\A}(t)=0.$$ 
\end{proposition}

\begin{proof}
Use Corollary \ref{cor:DC-const} (viewing as $k=1$) to reduce the problem to the case that $\A =  \scF \sqcup \scT$ with $\scF$ and $\scT \ne \emptyset$ consist of only coloops and loops, respectively. 
Then apply Proposition \ref{prop:gen-quasi-gcd2}(1).
\end{proof}

\begin{remark}
\label{rem:second-proof}
There is a neater proof: fix $\alpha\in\A^{\tor}$ and break $f^1_{\A}(t)$ into two summations with one of them is taken over $\scB \subseteq \A, \alpha\in \scB$.

\end{remark}

\begin{example}
\label{ex:caculation}
Let $\Gamma=\Z^2\oplus\Z/4\Z$, $\A=\{\alpha, \beta, \gamma\}\subsetneq\Gamma$ with 
$\alpha=(2, 2, \overline{1})$, $\beta=(0, 2, \overline{3})$ and $ \gamma=(0, 0, \overline{3})$. 
Then $\rho_\A=\rho_{\A\smallsetminus\{\gamma\}}=8$, and 
$$
\chi^{\quasi}_{\A}(q) = 
\begin{cases}
0 \quad\mbox{ if $\gcd(q,8)=1$}, \\
q^2  \quad\mbox{ if $\gcd(q,8)=2$}, \\
3q^2-4q+4 \quad\mbox{ if $\gcd(q,8)=4$}, \\
3q^2-12q+12  \quad\mbox{ if $\gcd(q,8)=8$}.
\end{cases}
$$
$$
\chi^{\quasi}_{\A\smallsetminus\{\gamma\}}(q) = 
\begin{cases}
q^2-2q+1 \quad\mbox{ if $\gcd(q,8)=1$}, \\
2q^2-4q+4  \quad\mbox{ if $\gcd(q,8)=2$}, \\
4q^2-8q+8  \quad\mbox{ if $\gcd(q,8)=4$}, \\
4q^2-16q+16 \quad\mbox{ if $\gcd(q,8)=8$}.
\end{cases}
$$
Note that $(\A\smallsetminus\{\gamma\})(\R)=\scL(\R)^X$, where $\scL(\R)=\{\{2x+2y+z=0\}, \{2y+3z=0\}, \{z=0\}\}\subseteq\R^3$ and $X=\{z=0\}$, which can also be identified with the 
integral arrangement $\{\{x+y=0\}, \{y=0\}\}$ in $\R^2$. 
In either way, 
$$\chi_{(\A\smallsetminus\{\gamma\})(\R)}(t)=f^1_{\A\smallsetminus\{\gamma\}}(t)=t^2-2t+1.$$ 
 \end{example}

\noindent
\textbf{Acknowledgements:} 
 The author is greatly indebted to Professor Masahiko Yoshinaga for many stimulating conversations, helpful suggestions and for his active interest in the publication of the paper.
He also gratefully acknowledges the support of the scholarship program of 
the Japanese Ministry of Education, Culture, Sports, Science, and Technology 
(MEXT) under grant number 142506. 
 \bibliographystyle{alpha} 
\bibliography{references}

\end{document}